\documentclass[a4paper,10pt]{amsart}

\usepackage{amsmath}
\usepackage{amssymb}
\usepackage{amsthm}
\usepackage{verbatim}
\usepackage[all]{xy}
\usepackage{enumerate}
\usepackage{setspace}

\usepackage[all]{xy}

\usepackage{color}

\newtheorem{thm}{Theorem}[section]

\newtheorem{lem}[thm]{Lemma}

\theoremstyle{definition}

\newtheorem{dfn}[thm]{Definition}

\numberwithin{equation}{section}
\newcommand{\cHil}{\mathcal{H}}
\newcommand{\cK}{\mathcal{K}}

\newcommand{\cM}{\mathcal{M}}
\newcommand{\cH}{\mathcal{H}}

\newcommand{\sfG}{\mathsf{G}}

\newcommand{\cI}{\mathcal{I}}

\newcommand{\cB}{\mathcal{B}}

\newcommand{\bc} {\Bbb C}

\newcommand{\ot}{\otimes}

\newcommand{\cR}{\mathcal{R}}

\newcommand{\bM}{\mathsf{M}}
\newcommand{\bA}{\mathsf{A}}

\newcommand{\boldv}{{\bf v}}
\newcommand{\boldw}{{\bf w}}

\newcommand{\cW}{\mathcal{W}}
\newcommand{\cWmin}{\mathcal{W}_{{\rm min}}}
\newcommand{\cWred}{\mathcal{W}_{{\rm red} }}

\newcommand{\Link}{\textsf{Link}}
\newcommand{\Star}{\textsf{Star}}

\newcommand{\cQ}{\mathcal{Q}}
\newcommand{\cE}{\mathcal{E}}
\newcommand{\cP}{\mathcal{P}}
\newcommand{\bN}{\mathsf{N}}

\newcommand{\cU}{\mathcal{U}}
\newcommand{\cV}{\mathcal{V}}
\newcommand{\bP}{\mathsf{P}}

\title[Connes embeddability of graph products]{Connes embeddability of graph products}

\date{\today}

\author{Martijn Caspers}
 \address{M. Caspers, Fachbereich Mathematik und Informatik der Universit\"at M\"unster,
Einsteinstrasse 62,
48149 M\"unster, Germany}
\email{martijn.caspers@uni-muenster.de}

\thanks{MC is supported by the grant SFB 878 ``{\it Groups, geometry and actions}''.
}

\begin{document}

\maketitle

\begin{abstract}
We prove that the Connes embedding problem is stable under graph products.
\end{abstract}

Graph products form a group theoretical construction generalizing free products by adding commutation relations that are dictated by a graph. The construction was first considered by Green in her thesis \cite{Green} and important examples of graph products arise as right angled Coxeter groups and right angled Artin groups. The formal definition is as follows.

\begin{dfn}
Let $\Gamma$ be a graph with vertex set $V\Gamma$ and edge set $E \Gamma$. We may assume that $\Gamma$ has no double edges and no loops, i.e. $(v,v) \not \in E\Gamma$.
 For $v \in V\Gamma$ let $\sfG_v$ be discrete group. Let $\sfG_\Gamma$ be the graph product group which is the discrete group freely generated by $\sfG_v, v \in V\Gamma$ subject to the relation $sts^{-1}t^{-1} = 1$ whenever $s \in \sfG_v$ and $t \in \sfG_w$ with $(v,w) \in E\Gamma$.
\end{dfn}

  Many stability properties of graph products have recently been found by various authors. For example: soficity \cite{CioHolRee2}, Haagerup property \cite{AntDre} (or \cite{CasFim}), residual finiteness \cite{Green}, rapid
decay \cite{CioHolRee}, linearity \cite{HsuWis} and many other properties, see e.g. \cite{HerMei}, \cite{AntMin}, \cite{Chi}.

Also in \cite{CasFim} the operator algebraic graph product was defined. The latter associates to a graph of von Neumann algebras $\bM_v, v \in V\Gamma$ equipped with faithful normal states a new von Neumann algebra $\bM_\Gamma$. The construction -- generalizing free products -- naturally has the property that in case  each $\bM_v$ is the group von Neumann algebra of a discrete group $\sfG_v$ then $\bM_\Gamma$ is isomorphic to the group von Neumann algebra of $\sfG_\Gamma$. In this paper we show that Connes embeddability  for von Neumann algebras is preserved by graph products (see Section \ref{Sect=Prelim} for further definitions).

Previous results concerning stability of the Connes embedding problem for {\it free} products were found by Popa \cite{Popa} and Isono--Houdayer \cite{HouIso}. In \cite{DykemaBrownJung} Brown, Dykema and Jung also found the free entropy dimension of free products. Here we  take the approach based on moment formulae that were given by Speicher in \cite{SpeicherLetter} (see also \cite{BozSpe}) partly inspired by  Junge \cite{Jun} and Nou  \cite{Nou}  -- also proving the result for free products.

\vspace{0.3cm}

\noindent {\bf Structure.} In Section \ref{Sect=Prelim} we recall preliminaries on ultraproducts of von Neumann algebras and graph products. Section \ref{Sect=Speicher} contains combinatorial arguments to show that graph products of the von Neumann algebras of Voiculescu's free Gaussian functor \cite{Voic} are embeddable. Section \ref{Sect=GraphEmbed} proves that graph products preserve the Connes embedding problem.

\subsection*{Acknowledgements} The author thanks Marius Junge for useful discussions and the referee for useful comments and reference \cite{SpeicherLetter}.

\section{Preliminaries}\label{Sect=Prelim}

\subsection{Von Neumann algebras}
For background on von Neumann algebras we refer to Takesaki's book \cite{Tak}. We may and will assume that every von Neumann algebra is represented on its standard GNS space. We shall assume that every von Neumann algebra is separable. Let $(\cM_i, \varphi_i)_{i \in I}$ be an indexed set of von Neumann algebras $\cM_i$ with faithful normal states $\varphi_i$. Let $\cU$ be a free ultrafilter on $\cI$. Let $\prod_{i, \cU} \cM_i$ be the Raynaud ultraproduct of von Neumann algebras which can canonically be identified with $(\prod_{i, \cU} (\cM_i)_\ast)^\ast$, see \cite{Ray} for details. As in \cite{Ray} we shall write $(x_i)^\bullet$ for a bounded sequence $(x_i)_{i \in I}$ with $x_i \in \cM_i$ identified within $\prod_{i, \cU} \cM_i$. Such sequences form a $\sigma$-weakly dense $\ast$-subalgebra of $\prod_{i, \cU} \cM_i$. We use analogous notation for an ultraproduct of states. Let $e$ be the support projection of the ultraproduct state $(\varphi_i)^\bullet$. The ultraproduct von Neumann algebra $\prod_{i, \cU} [\cM_i, \varphi_i]$ is defined as the corner algebra $e (\prod_{i, \cU} \cM_i) e$. In case each $(\cM_i, \varphi_i)$ is the hyperfinite II$_1$ factor $\cR$ equipped with its unique normal faithful tracial state $\tau$ we set for the resulting algebra $\mathcal{R}^\cU := \prod_{i, \cU} [\cR, \tau]$. We can now state the following conjecture/problem due to A. Connes.

\vspace{0.3cm}

\noindent {\bf Connes embedding problem:} Every separable II$_1$ factor $\cM$ embeds into $\cR^\cU$ for some free ultrafilter $\cU$.

\vspace{0.3cm}

\noindent Here an embedding means that there exists an injective normal $\ast$-homomorphism $\cM \hookrightarrow \cR^\cU$.
As the Connes embedding problem is open we shall say that a von Neumann algebra $\cM$ is {\it Connes embeddable} if an embedding into $\mathcal{R}^\cU$ exists. We refer to \cite{BrOz}, \cite{LupVal}, \cite{Oza} for more background and to \cite{BraColVer} for recent examples of  Connes embeddable von Neumann algebras.

\vspace{0.3cm}

\noindent Let again $(\cM_i, \varphi_i)_{i \in I}$ be an indexed set of von Neumann algebras $\cM_i$ with faithful normal states $\varphi_i$. Assume that each $\varphi_i$ is tracial. Let $\cU$ be a free ultrafilter on $I$. Let $\mathcal{I}_\cU$ be the $\sigma$-weakly closed ideal in $\prod_{i, \cU} \cM_i$ generated by all elements $(x_i)^\bullet$ satisfying $\lim_{i, \cU} \tau(x_i^\ast x_i) = 0$. The Ocneanu ultraproduct $\cM^\cU$ is defined as $\left( \prod_{i, \cU} \cM_i \right) \slash \mathcal{I}_\cU$. The Ocneanu ultraproduct $\cM^\cU$ is isomorphic to the ultraproduct $\prod_{i, \cU} [\cM_i, \varphi_i]$, see \cite[Corollary 3.28]{AndoHaagerup}.

\subsection{Graph products} We refer to \cite{Green} and \cite{CasFim} for the following results.
Let $\Gamma$ be a simplicial graph with vertex set $V\Gamma$ and edge set $E\Gamma$. Simplicial means that $\Gamma$ does not have double edges and for every $v \in V\Gamma$ we have $(v,v) \not \in E\Gamma$. We presume $\Gamma$ is non-oriented so that $(v,w) \in E\Gamma$ whenever $(w,v) \in E\Gamma$. In this paper we shall assume that $V\Gamma$ is countable so that the graph product of separable von Neumann algebras is again separable.   For $v \in V\Gamma$ we set $\Link(v) = \{ w \in V\Gamma \mid (w,v) \in E\Gamma\}$. We set $\Star(v) = \Link(v) \cup \{ v \}$.

 A word is a string of vertices and a word $\boldw = w_1 \ldots w_n$ is called {\it reduced} if the following property holds:
 if $w_i =  w_j$ with $i < j$ then there exists a $i < k < j$ such that $w_k \not \in \Link(w_i)$.
  We let $\cWred$ be the reduced words.
 We say that two words $\boldv$ and $\boldw$ are \textit{equivalent} if they are equivalent modulo the equivalence relation generated by the two relations:

\begin{equation}\label{Eqn=Equivalences}
\begin{split}
{\rm I}\quad(v_1, \ldots, v_i, v_{i+1}, \ldots, v_n) \simeq (v_1, \ldots, v_i, v_{i+2}, \ldots, v_n) \qquad& {\rm if} \quad  v_i = v_{i+1},\\
{\rm II}\quad(v_1, \ldots, v_i, v_{i+1}, \ldots, v_n) \simeq (v_1, \ldots, v_{i+1}, v_{i}, \ldots, v_n) \qquad& {\rm if} \quad v_i \in \Link(v_{i+1}).
\end{split}
\end{equation}
Moreover, we say that two words $\boldv$ and $\boldw$ are \textit{type ${\rm I}$ equivalent} if they are equivalent modulo the sub-equivalence relation generated by relation ${\rm I}$. We define the notion \textit{type ${\rm II}$ equivalent} in the analogous way.

Every word is equivalent to a reduced word \cite[Lemma 1.3]{CasFim}. Out of every equivalence class of words we may therefore pick a single distinguished reduced word, which we call minimal. We let $\cWmin$ be the set of minimal words and $\cWred$ be the set of reduced words.

\vspace{0.3cm}

\noindent Let $\bM_v, v \in V\Gamma$ be von Neumann algebras with normal faithful states $\varphi_v$. We set $\bM_v^\circ$ for the kernel of $\varphi_v$.
We  define the graph product von Neumann algebra in an explicit way, see \cite[Section 2]{CasFim}. Let $\cH_v$ be the Hilbert space on which $\bM_v$ acts and let $\xi_v \in \cH_v$ be a distinguished unit vector such that $\varphi_v(\: \cdot \:) = \langle \: \cdot \: \xi_v, \xi_v \rangle$. We let $\cH_v^\circ$ be the subspace of $\cH_v$ consisting of vectors orthogonal to $\xi_v$.  For a word $\boldv = v_1 \ldots v_n$ we set $\cHil_\boldv = \cH_{v_1}^\circ \otimes \ldots \otimes \cH_{v_n}^\circ$. By Lemma \cite[Lemma 1.3]{CasFim}  we see that if $\boldv \in \cWred$ is equivalent to $\boldw \in \cWred$ then there exists a uniquely determined unitary map,
\begin{equation}\label{Eqn=QMap}
\cQ_{\boldv, \boldw}: \cHil_{\boldv} \rightarrow \cHil_{\boldw}: \xi_{v_1} \otimes  \ldots \otimes \xi_{v_n} \mapsto
\xi_{v_{\sigma(1)}} \otimes  \ldots \otimes \xi_{v_{\sigma(n)}},
\end{equation}
where $\sigma$ is as in \cite[Lemma 1.3 (4)]{CasFim}.
Since every $\boldv \in \cWred$ has a unique minimal form $\boldv'$ we may simply write $\cQ_{\boldv}$ for $\cQ_{\boldv, \boldv'}$. 

\vspace{0.2cm}
\noindent Define the \textit{graph product Hilbert space} $(\mathcal{H},\Omega)$ by:
$$\mathcal{H}=\bc\Omega\oplus\bigoplus_{\boldw \in \cW_{\text{min}}}\mathcal{H}_{\boldw}.$$
Here $\Omega$ is a distinguished unit vector. The graph product state $\varphi_\Gamma$ is the vector state given by $\Omega$.
For $v\in V\Gamma$, let $\cW_v$ be the set of minimal reduced words $\boldw$ such that the concatenation $v\boldw$ is still reduced and write $\cW_v^c=\cW_{\text{min}}\setminus\cW_v$. Define
$$\mathcal{H}(v)=\bc\Omega\oplus\bigoplus_{\boldw\in\cW_v}\mathcal{H}_\boldw.$$
We define the isometry $U_v\,:\,\cH_v\ot\cH(v)\rightarrow\cH$ in the following way:
$$\begin{array}{llcl}
U_v\,:\,&\cH_v\ot\cH(v) &\longrightarrow& \cH\\
&\xi_v\ot\Omega &\overset{\simeq}{\longrightarrow}&\Omega\\
&\cH_v^{\circ}\ot\Omega&\overset{\simeq}{\longrightarrow}&\cH_v^{\circ}\\
&\xi_v\ot\cH_\boldw&\overset{\simeq}{\longrightarrow}&\cH_\boldw\\
&\cH_v^{\circ}\ot\cH_\boldw&\overset{\simeq}{\longrightarrow}&\cQ_{v\boldw}(\cH_v^{\circ}\ot\cH_\boldw)
\end{array}$$
Here the actions are understood naturally. Observe that, for any reduced word $\boldw$, the word $v\boldw$ is not reduced if and only if $\boldw$ is equivalent to a reduced word that starts with $v$. It follows that $U_v$ is surjective, hence unitary.
Define, for $v\in V\Gamma$, the faithful unital normal $*$-homomorphism $\lambda_v\,:\,\mathcal{B}(\cH_v)\rightarrow\mathcal{B}(\cH)$ by
$$\lambda_v(x)=U_v(x\ot 1)U_v^*\quad\text{for all}\quad x\in\cB(\cH_v).$$
Observe the $\lambda_v$ intertwines the vector states $\omega_{\xi_v}$ and $\omega_{\Omega}$. The graph product $\bM_\Gamma$ is defined as the von Neumann algebra generated by $\cup_{v \in V\Gamma} \lambda_v(\bM_v)$. We shall identify $\bM_v$ as a von Neumann subalgebra of $\bM_\Gamma$ and omit $\lambda_v$ in the notation.
If each $\bM_v, v \in V\Gamma$ is a II$_1$-factor, then so is $\bM_\Gamma$, see \cite[Corollary 2.29]{CasFim}. We shall use the fact that $\bM_v$ and $\bM_w$ commute whenever $(v,w) \in E\Gamma$ without further reference, see \cite[Section 2]{CasFim}. Finally, we mention that graph products satisfy a universal property for which we refer to \cite[Proposition 2.22]{CasFim}.

\section{Graph products and Voiculescu's free Gaussian functor}\label{Sect=Speicher}

In this section we consider graph products of Voiculescu's factors \cite{Voic} (which are in fact free group factors) and find a moment formula.

\subsection{Preliminaries on partitions}
We let $\cP(1, \ldots, n)$ be the set of all partitions of the set $\{ 1, \ldots, n\}$. We let $\cP_2(1, \ldots, n)$ be the set of all pair partitions, meaning that every equivalence class consists exactly of 2 elements. In particular the latter set is empty if $n$ is odd.  Let $\cV \in \cP_2(1, \ldots, n)$ and write $\cV = \{ (e_1, z_1), \ldots, (e_r, z_r) \}, 2r=n, e_i < z_i$. Let $I(\cV)$ be the set of all pairs $(k,l)$ such that $e_k < e_l < z_k < z_l$. Let $\Gamma$ be a simplicial graph and let $\boldv$ be a (not nessecarily reduced) word of length $n$. We let $\cP(\boldv)$ be the set of partitions $\cV$ of $(1, \ldots, n)$ with the property that if $k$ and $l$ are equivalent in $\cV$ then this implies that $v_k = v_l$.
 We let $\cP_2(\boldv)$ be the subset of $\cP(\boldv)$ of pair partitions. For $\cV \in \cP_2(\boldv)$ we define $I_{\Gamma}(\cV)$  as the subset of $I(\cV)$ consisting of all pairs $(k,l)$ such that $e_k < e_l < z_k < z_l$ and $(v_{e_k}, v_{e_l}) \not \in E\Gamma$. A pair partition $\cV$ is called admissible if $I(\cV)$ is empty and it is called $\Gamma$-admissible if $I_{\Gamma}(\cV)$ is empty. We denote $\cP_{2, {\rm nc}, \Gamma}(\boldv)$ for the subset of $\cP_{2}(\boldv)$  consisting of all $\Gamma$-admissible partitions (these can be considered as ``non-crossing partitions up to permutations coming from edges of $\Gamma$''). Let $\cV \in \cP(\boldv)$ we say that a sequence of indices $(i_1, \ldots, i_n)$ is of class $\cV$ if the following property holds: $i_k = i_l$ and $v_k = v_l$ if and only if $k$ and $l$ are equivalent in $\cV$.

\subsection{Voiculescu's free Gaussian functor, graph products and a moment formula}\label{Sect=Gaussian}
We recall the construction of the factors constructed in \cite{Voic}.
 Let $\cH$ be $\mathbb{C}^2$ be the two dimensional Hilbert space with orthonormal basis vectors $f_1$ and $f_2$. Let $\cK = \oplus_{n=1}^\infty \cH^{\otimes n} \oplus \mathbb{C} \cdot \Omega$ where $\Omega$ is a distinguished (vacuum) unit vector. For $f \in \cH$ let $a^\ast(f)$ be the creation operator and $a(f)$ be the annihilation operator:
\begin{equation}\label{Eqn=CreationOperator}
\begin{array}{rlll}
a^\ast(f) \xi & = & f \otimes \xi, & \xi \in \cH^{\otimes n},\\
a(f) \Omega & = & 0, &\\
a(f) g \otimes \xi & = & \langle g, f \rangle \xi, & \xi \in \cH^{\otimes n}.
\end{array}
\end{equation}
Set,
\begin{equation}\label{Eqn=GaussianGenerators}
g_1 = a^\ast(f_1) + a(f_1), \qquad g_2 = a^\ast(f_2) + a(f_2).
\end{equation}
Let $\bP$ be the von Neumann algebra generated by $g_1$ and $g_2$. In fact    $\bP$ is equal to the free product of the von Neumann algebras generated by $g_1$ and $g_2$. Let $\Gamma$ be a simplicial graph. We equip the notions introduced so far in this subsection with a $v \in V\Gamma$ to distinguish them for different vertices. For every $v \in V\Gamma$ let $\bP_v := \bP$ and let $\cK_v := \cK$ be the Hilbert space it acts on. For $\bP_v$ we denote $g_{1,v}$ and $g_{2,v}$ for the generators defined in \eqref{Eqn=GaussianGenerators}. We set $a_{i,v}^\ast$ for $a^\ast(f_{i,v})$ as in \eqref{Eqn=CreationOperator}.
Let $\bP_\Gamma$ be the graph product von Neumann algebra of the $\bP_v, v \in V\Gamma$ with graph product state $\tau_\Gamma$.

\begin{thm}\label{Thm=GaussianMoments}
Let $d_{j} \in \{ g_{i,v} \mid v \in V\Gamma, i \in \{1,2\} \}, 1 \leq j \leq n$. Then,
\begin{itemize}
\item\label{Item=TraceI} $\tau_\Gamma( d_1 d_2 \ldots d_n)$ is 0 in case $n$ is odd.
\item Let $v_j$ be such that $d_j = g_{i,v_j}$ for either $i=1$ or $i=2$ and suppose that $n$ is even, say $n=2r$. Then,
\[
\tau_\Gamma( d_1 d_2 \ldots d_n) =
\# \cP_{2, {\rm nc}, \Gamma}(\boldv).
\]
\end{itemize}
\end{thm}
\begin{proof}
Let $\cK_\Gamma$ be the graph product Hilbert space on which $\bP_\Gamma$ acts. Let $\cW_0$ be the set of all words that are I-equivalent to the words in $\cWmin$. For a word $\boldv \in \cW_0$ we recall that we have set $\cH_{\boldv} = \cH_{v_1}^\circ \otimes \ldots \otimes \cH_{v_n}^\circ$. We say that the vectors in the latter space have symbol $\boldv$.
We have $\cK_{\Gamma} \simeq \oplus_{\boldv \in \cW_0 } \cH_{\boldv}$ by associativity of tensor products. We describe the actions of the creation and annihilation operators:
\[
a_{i,v_j}^\ast \xi_1 \otimes \ldots \otimes \xi_n =
 \xi_{k_1} \otimes \ldots \otimes \xi_{k_s} \otimes e_{i,v_j} \otimes \xi_{k_{s+1}} \otimes \ldots \otimes \xi_{k_n},
\]
where $\xi_1 \otimes \ldots \otimes \xi_n$ has symbol $\boldw \in \cW_0$ and $w_{k_1} \ldots w_{k_{s}} v_{j} w_{k_{s+1}} \ldots w_{k_n}$ is the unique word in $\cW_0$ that is  II-equivalent to $v_j w_{1} \ldots w_{n}$ with $w_{k_{s}} \not = v_j$  and which satisfies the property that in case $w_{k_i} = w_{k_{i+1}}$ we must have $k_{i+1} = k_i + 1$.
 Furthermore for a vector $\xi_1 \otimes \ldots \otimes \xi_n$ with symbol $\boldw$ we find:
\[
\begin{split}
a_{i,v_j} \xi_1 \otimes \ldots \otimes \xi_n =
\left\{
\begin{array}{ll}
 \langle  \xi_{k_s }, f_i \rangle \xi_{k_1} \otimes \ldots \otimes \widehat{\xi}_{k_s} \otimes \ldots \otimes \xi_{k_n}    & \textrm{in case ($\ast$) below holds},\\
0  &\textrm{else},\\
\end{array}
\right.
\end{split}
\]
where ($\ast$) means that there exist $k_1 \ldots k_n$ determined by the property that  $v_j w_{k_1}\ldots \widehat{w}_{k_s} \ldots w_{k_n}$ is II-equivalent to $w_{1} \ldots w_{n}$ (so in particular $w_{k_s} = v_j$), $w_{k_{s-1}} \not = v_j$,  $w_{k_1}\ldots \widehat{w}_{k_s} \ldots w_{k_n}$ is in $\cW_0$ and in case $w_{k_i} = w_{k_{i+1}}$ we must have $k_{i+1} = k_i + 1$.

The action of  $d_1 \ldots d_n = \sum_{(k_1, \ldots, k_n) \in \{ 1, \ast \}^n } a_{i_1, v_1}^{k_1} \ldots a_{i_n, v_n}^{k_n}$ on $\Omega$ is thus described by sums of $n$ creation/ annihilation operators. If the trace,
\[
\tau_\Gamma( a_{i_1, v_1}^{k_1} \ldots a_{i_n, v_n}^{k_n} ) = \langle a_{i_1, v_1}^{k_1} \ldots a_{i_n, v_n}^{k_n} \Omega, \Omega \rangle,
\]
is non-zero,
then we must have exactly $\frac{n}{2}$ creation operators and $\frac{n}{2}$ annihilation operators occuring in $a_{i_1, v_1}^{k_1} \ldots a_{i_n, v_n}^{k_n}$ and in particular $n$ must be even. This proves the first statement. So let $n$ be even and let $(k_1, \ldots, k_n) \in \{ 1, \ast\}^n$ be such that exactly half of the terms equals 1 and the other half equals $\ast$. We associate a pair partition to any term $a_{i_1, v_1}^{k_1} \ldots a_{i_n, v_n}^{k_n}$ with non-zero trace in the following way.
We connect $s<r$ if $k_s = 1$ and $k_r = \ast$, $v_s = v_r$ and $a_{i_s, v_s}$ annihilates the vector that was created by $a^\ast_{i_r, v_r}$. Call the resulting pair partition $\cV_{k_1, \ldots, k_n}$.

\vspace{0.3cm}

\noindent {\bf Claim:} Let $\mathcal{S}$ be the set $\left\{  \cV_{k_1, \ldots, k_n} \mid (k_1, \ldots, k_n) \in \{1, \ast\}^n \textrm{ and } \tau( a_{i_1, v_i}^{k_1} \ldots a_{i_n, v_n}^{k_n} ) \not = 0 \right\}$. Then we have $\mathcal{S} = \cP_{2, {\rm nc}, \Gamma}(\boldv)$.

\noindent {\it Proof of the Claim:} $\subseteq$. Suppose the inclusion does not hold. Then there exist a partition $\cV \in \mathcal{S}$ with $e_a < e_b < z_a < z_b$ such that $(e_a, z_a) \in \cV, (e_b, z_b) \in \cV$ and $(v_a, v_b) \not \in E\Gamma$. This means that
\[
a_{i_{e_b}, v_{e_b}}^{k_{e_b}} (a_{i_{e_b} +1, v_{e_b +1}}^{k_{e_b +1} } \ldots a_{i_{n}, v_{n}}^{k_{n} } \Omega ) = 0,
\]
which contradicts that $\cV \in \mathcal{S}$.

$\supseteq$. Again suppose that the inclusion does not hold. Then there exists a pair partition $\cV \in \mathcal{P}_{2, nc, \Gamma}$ such that
\begin{equation}\label{Eqn=TauMakesZero}
\tau( a_{i_1, v_1}^{k_1} \ldots a_{i_n, v_n}^{k_n} ) = 0,
 \end{equation}
 where the $k_i$ are the unique indices determined by:
\[
k_i =
\left\{
\begin{array}{ll}
1 & \textrm{if } i = z_a \textrm{ for some } (e_a, z_a) \in \cV \textrm{ with } e_a < z_a,  \\
\ast & \textrm{if } i = e_a \textrm{ for some } (e_a, z_a) \in \cV \textrm{ with } e_a < z_a.
\end{array}
\right.
\]
In particular, $\cV = \cV_{k_1, \ldots, k_n}$. But if \eqref{Eqn=TauMakesZero} holds and taking into account that $a_{i_1, v_1}^{k_1} \ldots a_{i_n, v_n}^{k_n}$ creates as many vectors as it annihilates (i.e. exactly  half of the $k_i$'s equal $\ast$), this shows that we must have
\[
a_{i_{e_b}, v_{e_b}}^{k_{e_b}} (a_{i_{e_b} +1, v_{e_b +1}}^{k_{e_b +1} } \ldots a_{i_{n}, v_{n}}^{k_{n} } \Omega ) = 0,
\]
for some index $e_b$ for which $k_{e_b} = 1$. This can only happen if there were indices $e_a, z_a, z_b$ such that $e_a < e_b< z_a< z_b$ and $(e_a, z_a) \in \cV, (e_b, z_b) \in \cV$ such that $(v_a, v_b) \not \in E\Gamma$. This contradicts that $\cV \in \cP_{2, nc, \Gamma}(\boldv)$.

\vspace{0.3cm}

\noindent {\it Remainder of the proof:} We now have, putting $2r = n$,
\[
\begin{split}
\tau_\Gamma(d_1 \ldots d_n) = &
\sum_{ \cV = \{ (e_1, z_1), \ldots, (e_r, z_r) \} \in \cP_{2, {\rm nc}, \Gamma}(\boldv) }  \langle a_{i, v_{1}}^{k_1} \ldots a_{i_n, v_n}^{k_n} \Omega, \Omega \rangle \\
= & \sum_{ \cV = \{ (e_1, z_1), \ldots, (e_r, z_r) \} \in \cP_{2, {\rm nc}, \Gamma}(\boldv) }  \prod_{k=1}^{r} \langle f_{i_{e_k}}, f_{i_{z_k}} \rangle \\
= & \# \cP_{2, {\rm nc}, \Gamma}(\boldv).
\end{split}
\]
This concludes the theorem.
\end{proof}

Finally, we recall the following fact (see \cite{Voic} and \cite[Corollary 2.29]{CasFim}).

\begin{thm}\label{Thm=GaussianFactor}
$\bP$ is a II$_1$-factor. In particular $\bP_v := \bP$ and hence $\bP_\Gamma$ is a II$_1$-factor.
\end{thm}

\subsection{The Speicher central limit theorem} We state a version of Speicher's central limit theorem \cite{Spe} adapted to graph products. For a simplicial graph $\Gamma$ and some index set $I$ we let $s: I \times V\Gamma \times I \times V\Gamma \rightarrow \{ -1, 1 \}$ be a sign function. Let $\boldv$ be a word with letters in $V\Gamma$ and let $\cV \in \cP_2(\boldv)$. Let $n$ be the length of the word $\boldv$. We set,
\[
t(\cV) = \lim_{N \rightarrow \infty}  \frac{1}{ N^{n/2} } \sum_{\begin{array}{c}i_1, \ldots, i_n = 1 \\ (i_1, \ldots , i_n) \textrm{ is of class } \cV \end{array} }^N \left( \prod_{(a,b) \in I_{\Gamma}(\cV) } s(i_{e_a}, v_{e_a}, i_{e_b}, v_{e_b}) \right).
\]
The following theorem is proved in \cite[Theorem 1]{SpeicherLetter} (see also  \cite[Theorem 1]{Spe}).
\begin{thm}\label{Thm=Steps}
Let $\Gamma$ be a simplicial graph. Let $\mathcal{A}$ be a $\ast$-algebra generated by self-adjoint elements $b_{i,v}, i \in \mathbb{N}, v \in V\Gamma$.   Presume that  
there is a sign function $s: \mathbb{N} \times V\Gamma \times \mathbb{N} \times V\Gamma \rightarrow \{ -1 , 1\}$ satisfying $s(i, v, j, w) = 1$ whenever $(v,w) \in E \Gamma$ such that:
    \begin{equation}\label{Eqn=Commutation}
     b_{i,v} b_{j,w} = s(i,v,j,w) b_{j,w} b_{i,v}.
    \end{equation}
For $v \in V\Gamma$ and $N \in \mathbb{N}$ set
\[
S_{N,v} := \frac{b_{1,v} + \ldots + b_{N,v}}{\sqrt{N}}.
\]
Assume that for every (not necessarily reduced) word $\boldv = v_1 \ldots v_n$  the quantity $t(\cV)$ exists for every $\cV \in \cP_2(\boldv)$. Then there exists a state $\rho$ on $\mathcal{A}$ that satisfies the following:
\begin{equation}\label{Eqn=SomeState}
\lim_{N \rightarrow \infty} \varphi( S_{N, v_1}  \ldots S_{N, v_n}  ) =
\left\{
\begin{array}{ll}
0 & n \textrm{ is odd.} \\
\sum_{\cV \in \cP_2(\boldv), \cV = \{ (e_1, z_1),\ldots, (e_r, z_r) \} } t(\cV)
   & n=2r,
\end{array}
\right.
\end{equation}
$\rho(b_{i,v}^2) = 1$ and  $\rho(b_{i,v} b_{j,w}) = 0$ in case $(i,v) \not = (j,w)$.
\end{thm}

 Next we show that $t(\cV)$ in the previous theorem can be computed almost everywhere. In order to do so, let $\geq$ be some linear order on $\mathbb{N} \times V\Gamma$. Naturally the symbol $>$ stands for $\geq$ but not equal.

\begin{lem}\label{Lem=Proba}
For $\Gamma$ a simplicial graph, let ${\bf s} = (s(v,i,w,j))_{v,w \in V\Gamma, i,j \in \mathbb{N}}$ be an infinite random matrix with the properties:
\begin{enumerate}
\item\label{Item=RandomMatrixI} $s(i,v, j, w) = s(j,w, i,v)$ for every $i,j \in I$ and $v,w \in V\Gamma$,
\item $s(i,v, j,w) = 1$ whenever $(v,w) \in E\Gamma$,
\item $s(i,v,j,w)$ with $(i,v) > (j, w)$ are independent,
\item\label{Item=RandomMatrixV} ${\rm prob}(s(i,v,j,w) = 1) = p$ and ${\rm prob}(s(i,v,j,w) = -1) = q := 1-p$, whenever $(v,w) \not \in E\Gamma$.
\end{enumerate}
Let $\boldv$ be a (not necessarily reduced) word with letters in $V\Gamma$. Then, for almost every ${\bf s}$ we have for all $\cV \in \cP_{2}(\boldv)$ that
\[
t(\cV) = (p-q)^{ \# I_{\Gamma}(\cV) }.
\]
\end{lem}
\begin{proof}
As $s(i,v, j,w) = 1$ whenever $(v,w) \in E\Gamma$ we have (trivially) that ${\rm prob}( s(i,v, j,w) = 1) = 1$ whenever $(v,w) \in E\Gamma$. So the lemma follows from \cite[Theorem 2]{SpeicherLetter}.
\end{proof}

\subsection{Matrix models} We now prove that the von Neumann algebra $\bP_\Gamma$ introduced in Section \ref{Sect=Gaussian} is embeddable. For the moment assume that $\Gamma$ is a finite graph. Let $N \in \mathbb{N}$ and set $I_N = \{0, \ldots, N\}$. Let $s: I_N \times V\Gamma \times I_N \times V\Gamma \rightarrow \{ -1, 1 \}$ be a sign function satisfying the properties:
\begin{enumerate}
\item\label{Item=EpsilonI} $s( i, v, j, w ) = s( j, w, i, v )$;
\item\label{Item=EpsilonII} $s( i, v, i, v ) = -1$.
\end{enumerate}
We let $x_{i, v}$ be algebraic generators of an algebra $\mathcal{A}$ satisfying the following relations:
\[
x_{i,v } x_{j, w} - s(i,v,j,w) x_{j,w} x_{i,v} = 2 \delta_{i,j} \delta_{v,w}.
\]
In particular $x_{i,v}^2 = 1$ and it follows from these (anti-)commutation relations that $\mathcal{A}$ is finite dimensional. Fix a linear order on $I_N \times V\Gamma$. For $A \subseteq I_N \times V\Gamma$ we set,
\[
x_A := \prod_{(i,v) \in A} x_{i,v},
\]
where the product is taken with respect to the linear order. The sets $x_A$ form a basis of $\mathcal{A}$. We equip $\mathcal{A}$ with the $\ast$-structure given by $x_{i,v}^\ast = x_{i,v}$. Let $\varphi$ be the tracial function $\mathcal{A} \rightarrow \mathbb{C}$ defined by $\varphi(x_{A}) = \delta_{A, \emptyset}$. Then $\langle x, y \rangle = \varphi(y^\ast x)$ defines an inner product and hence a Hilbert space $L^2(\mathcal{A}, \varphi)$. We define partial isometries $a^\ast_{i,v}$ and $a_{i,v}$ by
\begin{equation}\label{Eqn=Create}
a^\ast_{i,v} x_A = \left\{
\begin{array}{ll}
x_{i,v} X_A & \textrm{ if } (i,v) \not \in A,\\
0 & \textrm{ if } (i,v) \in A,
\end{array}
\right.
\end{equation}
and
\begin{equation}\label{Eqn=Destroy}
a_{i,v} x_A = \left\{
\begin{array}{ll}
x_{i,v} X_A & \textrm{ if } (i,v) \in A,\\
0 & \textrm{ if } (i,v) \not \in A,
\end{array}
\right.
\end{equation}
Note that \eqref{Eqn=Create} is the adjoint of \eqref{Eqn=Destroy}. Then we set
\[
b_{i,v} = a_{i,v}^\ast + a_{i,v}.
\]
These operators satisfy the relations
\[
b_{i,v} b_{j,w} = s(i,v,j,w) b_{j,w} b_{i,v}.
\]
We set for $N \in \mathbb{N}$ even,
\[
S_{N,v,1} = \frac{1}{\sqrt{N}} \sum_{i=0}^{N-1} b_{2i,v}, \qquad S_{N,v,2} = \frac{1}{\sqrt{N}} \sum_{i=0}^{N-1} b_{2i+1,v},
\]
Now we assume that $s: I_\infty \times V\Gamma \times I_\infty \times V\Gamma \rightarrow \{-1, 1\}$ is an infinite random matrix with entries being independent identically distributed random variables subject to conditions \eqref{Item=EpsilonI} and \eqref{Item=EpsilonII} and such that whenever $(v,w) \not \in E\Gamma$ we have ${\rm prob}(s(i,v,j,w) = 1) = {\rm prob}(s(i,v,j,w) = -1) = \frac{1}{2}$. The following result is now a direct consequence of Theorem \ref{Thm=GaussianMoments}, Theorem \ref{Thm=Steps} and Lemma \ref{Lem=Proba}. Note that Theorem \ref{Thm=Steps} shows that there is some $\rho$ satisfying the moment formula of Theorem \ref{Thm=Steps}, but in fact the extra condition $b_{i,v}^2 = 1$ implies that $\rho$ is $\varphi$ above (see Theorem 3 of \cite{SpeicherLetter} and in particular the discussion after it).
\begin{thm}
Let $\Gamma$ be a finite simplicial graph and let $V\Gamma = \{ v_1, \ldots, v_n \}$.
For any $\ast$-polynomial $Q$ in $2 \# V\Gamma$ non-commutating variables we have
\begin{equation}\label{Eqn=GvsS}
\tau\left( Q(g_{1,v_1}, \ldots, g_{1, v_n}, g_{2,v_1}, \ldots, g_{2,v_n})  \right)
= \lim_{N \rightarrow \infty} \varphi\left( Q( S_{N, v_1, 1}, \ldots, S_{N, v_n, 1}, S_{N, v_1, 2}, \ldots, S_{N, v_n, 2} )  \right)
\end{equation}
for almost every $s$.
\end{thm}

Let $\cU$ be a free ultrafilter on $\mathbb{N}$. We wish to prove that $\bP_\Gamma$ (see Section \ref{Sect=Gaussian}) is Connes embeddable by defining an injective $\ast$-homomorphism $\Phi$ by setting
\[
\Phi\left(  Q(g_{1,v_1}, \ldots, g_{1, v_n}, g_{2,v_1}, \ldots, g_{2,v_n})    \right) =
\prod_{N, \cU}  Q(  S_{N, v_1, 1}, \ldots, S_{N, v_n, 1}, S_{N, v_1, 2}, \ldots, S_{N, v_n, 2} ).
\]
However, the entries of the ultra product on the right hand side may not be bounded. Therefore we need bounded cut-offs and at this point the argument is exactly the same as in \cite[Section 3]{Nou}. Meaning, let $C$ be any constant such that $\Vert g_{1, v_1} \Vert \leq C$. Let $\chi_{[-C, C]}$ be the characteristic function of the interval $[-C, C]$. We set $\widetilde{S}_{N, v_i, i} := \chi_{[-C, C]}(S_{N, v, i} ) S_{N, v, i}$. What is proved in \cite[Section 3]{Nou} is that \eqref{Eqn=GvsS} also holds if $S_{N, v_i, i}$ is replaced by $\widetilde{S}_{N, v_i, i}$. As a consequnce:

\begin{thm}\label{Thm=PGammaIsEmbeddable}
The graph product von Neumann algebra $\bP_\Gamma$ is Connes embeddable.
\end{thm}

\section{Embeddability of graph products}\label{Sect=GraphEmbed}
We prove the main result of this paper, namely that the Connes embedding problem is stable under graph products.
The following Lemma \ref{Lem=CEVSupport} is proved in the first paragraph of the proof of \cite[Lemma 7.18]{Jun}.

\begin{lem}\label{Lem=CEVSupport}
Let $\bN$ and $\bM$ be von Neumann algebras. Let $\cE: \bN \rightarrow \bM$ be a conditional expectation. Let $\omega$ be a normal state on $\bM$. Let $f \in \bN$ be the support of $\omega \circ \cE$ and let $e \in \bM$ be the support of $\omega$. Then $f$ commutes with every element in $e \bM e$.
\end{lem}

\begin{lem}\label{Lem=PointGraphProduct}
Let $\bM_v, v \in V\Gamma$ be von Neumann algebras with normal faithful tracial state $\tau_v$. Suppose that for every $v \in V\Gamma$ there are von Neumann algebras $\bA_{v,i}$ with normal faithful tracial states $\tau_{v,i}$ with a trace preserving embedding:
\[
\pi_v: \bM_v \rightarrow \prod_{i, \cU} [\bA_{v, i}, \tau_{v,i}],
\]
then there exists a trace preserving embedding,
\[
\pi_{\Gamma}: \bM_{\Gamma} \rightarrow \prod_{i, \cU} [ \bA_{\Gamma,i}, \tau_{\Gamma,i} ].
\]
\end{lem}
\begin{proof}
Let $j_{v,i}: \bM_{v,i} \rightarrow \bM_{\Gamma,i}$ be the natural trace preserving  embedding. The predual maps $(j_{v,i})_\ast: (\bM_{\Gamma,i})_\ast \rightarrow (\bM_{v,i})_{\ast}$ are contractive and hence induce a mapping in the ultraproduct $(j_v)_\ast:  \prod_{i, \cU} (\bM_{\Gamma,i})_\ast \rightarrow \prod_{i, \cU} (\bM_{v,i})_{\ast}$. We let
\[
j_v: \prod_{i, \cU} \bM_{v, i} \rightarrow \prod_{i, \cU} \bM_{\Gamma,i}
\]
be the dual of this mapping. Note that $\prod_{i, \cU} \bM_{v,i}$ is the dual of $\prod_{i, \cU} (\bM_{v,i})_\ast$ via the pairing $\langle (x_i)^\bullet, (\omega_i)^\bullet \rangle = \lim_{i, \cU} \omega_i(x_i)$ and therefore explicitly $j_v( (x_i)^\bullet ) = (j_{v,i}(x_i) )^\bullet$.

Let $e_v, v \in V\Gamma$ be the support projection of the ultraproduct state $(\tau_{v,i})^\bullet$. Let $f$ be the support projection of the ultraproduct state $(\tau_{\Gamma, i})^\bullet$. Note that $j_v$ identifies $\prod_{i, \cU} \bM_{v, i}$ as a subalgebra of $\prod_{i, \cU} \bM_{\Gamma, i}$ and $(\tau_{\Gamma, i})^\bullet$ restricts to $(\tau_{v,i})^\bullet$ \cite{CasFim}.

Recall that $e_v \prod_{i, \cU} \bM_{v,i}  e_v = \prod_{i, \cU} [\bM_{v,i}, \tau_{v,i}]$ and similarly $f \prod_{i, \cU} \bM_{\Gamma,i}  f = \prod_{i, \cU} [\bM_{\Gamma,i}, \tau_{\Gamma,i}]$. Set $\rho_v: \prod_{i, \cU} [\bM_{v,i}, \tau_{v,i}] \rightarrow \prod_{i, \cU} [\bM_{\Gamma,i}, \tau_{\Gamma,i}]$ by defining $\rho_v(e_v x e_v) = f j_v(e_v x e_v) f$, where $x \in \prod_{i, \cU} \bM_{v,i}$. By Lemma \ref{Lem=CEVSupport} $f$ commutes with the image of $j_v$, from which we conclude that $\rho_v$ is a $\ast$-homomorphism. Set $\alpha_v = \rho_v \circ \pi_v$. Note that $\alpha_v$ is faithful: indeed let $0 \leq x \in \bM_v$ be non-zero. Then $(\tau_{\Gamma,i})^\bullet (\alpha_v(x)) = \tau_v(x) \not = 0$. Hence $\alpha_v(x) \not = 0$.

We shall now verify the universal property \cite[Proposition 2.22]{CasFim}. Let $(v,w) \in E\Gamma$. For $x = (x_i)^\bullet$ and $y = (y_i)^\bullet$ in respectively $\prod_{i, \cU} \bM_{v,i}$ and $\prod_{i, \cU} \bM_{w,i}$ we have since $x_i y_i = y_i x_i$:
\[
\begin{split}
&\rho_v(e_v x e_v) \rho_{v'}(e_{v'} x e_{v'}) = f j_v(e_v x e_v) f j_w(e_w y e_w) f = f j_v(e_v x e_v)  j_w(e_w y e_w) f \\
= &
f (j_{v,i}( e_{v,i} x_i e_{v,i}) )^\bullet (j_{w,i}( e_{w,i} y_i e_{w,i}) )^\bullet f =
f  (j_{w,i}( e_{w,i} y_i e_{w,i}) )^\bullet (j_{v,i}( e_{v,i} x_i e_{v,i}) )^\bullet f,
\end{split}
\]
and the latter expression equals $\rho_v(e_v x e_v) \rho_{v'}(e_{v'} x e_{v'})$. So the images of $\rho_v$ and $\rho_w$ commute and hence the images of $\alpha_v$ and $\alpha_w$ commute.

Next, let $\boldv = v_1 \ldots v_n$ be a reduced word. For $1 \leq k \leq n$ let $a_k \in \bM_{v_k}^\circ$. Since $e_{v_k} \left( \prod_{i, \cU} \bA_{v_k, i} \right) e_{v_k}$ equals $\prod_{i, \cU} [\bA_{v_k,i}, \tau_{v_k, i}]$ we may approximate $\pi_{v_k}(a_k) \in \prod_{i, \cU} [\bA_{v_k, i}, \tau_{v_k, i} ]$ in the strong topology with a bounded net $(a_{k, i, s})^\bullet_{s \in S}$ where $a_{k, i, s} \in \bA_{v_k, i}$ (by Kaplansky's density theorem). Since $\pi_{v_k}$ is trace preserving it follows that $\lim_{s \in S} \lim_{i, \cU} \tau_{v_k, i}(a_{k, i, s}) = 0$. Therefore we may replace $a_{k, i, s}$ by $a_{k, i, s}^\circ := a_{k, i, s} - \tau_{v_k, i}(a_{k, i, s}) \in \bA_{v_k, i}^\circ$ and still have $(a_{k, i, s}^\circ)^\bullet_{s \in S} \rightarrow \pi_{v_k}(a_k)$ strongly. Then,
\[
\begin{split}
& \tau_{\Gamma}( \alpha_{v_1}(a_1) \ldots \alpha_{v_n}(a_n)  ) \\
= &\lim_{s \in S} \tau_\Gamma\left( f ( j_{v_1,i}( a^\circ_{1, i, s} )   )^\bullet  f \alpha_{v_2}(a_2) \ldots \alpha_{v_n}(a_n) \right) \\
= &\lim_{s_1 \in S_1} \ldots \lim_{s_n \in S_n} \tau_\Gamma\left( f ( j_{v_1,i}( a^\circ_{1, i, s_1} )   )^\bullet  f  \ldots  f ( j_{v_n,i}( a^\circ_{n, i, s_n} )   )^\bullet  f \right) \\
= &\lim_{s_1 \in S_1} \ldots \lim_{s_n \in S_n} \lim_{i, \cU} \tau_{\Gamma,i} \left(  j_{v_1,i}( a^\circ_{1, i, s_1} )       \ldots     j_{v_n,i}( a^\circ_{n, i, s_n} )    \right),
\end{split}
\]
which equals zero as $\tau_{\Gamma,i} \left(  j_{v_1,i}( a^\circ_{1, i, s_1} )       \ldots     j_{v_n,i}( a^\circ_{n, i, s_n} )    \right) = 0$. Hence we may apply \cite[Proposition 2.22]{CasFim} which concludes that there exists a trace preserving embedding $\pi_{\Gamma}: \bM_{\Gamma} \rightarrow \prod_{i, \cU} [\bA_{\Gamma,i}, \tau_{\Gamma,i}]$.
\end{proof}

\begin{lem}\label{Lem=MatrixFactor}
Let $\bM$ be a type II$_1$ factor with normal  faithful tracial state $\tau$. Consider $M_n(\mathbb{C})$ with normalized trace.  There exists a trace preserving embedding $\varphi_n: M_n(\mathbb{C}) \rightarrow \bM$.
\end{lem}
\begin{proof}
Let $p_1, \ldots, p_n$ be $n$ mutually orthogonal projections in $\bM$ with $\tau(p_n) = n^{-1}$. Since $\cM$ is a type II factor let $u_{i,j}, i \not = j$ be partial isometries with $u_{i,j} u_{i,j}^\ast = p_{i}, u_{i,j}^\ast u_{i,j} = p_j$. Put $u_{i,i} = p_i$. Let $e_{i,j}$ be the matrix units of $M_n(\mathbb{C})$. Then extending $\varphi_n: e_{i,j} \mapsto u_{i,j}$ linearly gives the required mapping.
\end{proof}

\begin{thm}\label{Thm=GraphQWEP}
Let $\Gamma$ be a countable simplicial graph. For every $v \in V\Gamma$ let $\bM_v$ be a II$_1$ factor with normal faithful tracial state $\tau_{v}$. The graph product $\bM_\Gamma$ is Connes embeddable if and only if for every $v \in V\Gamma, \bM_v$ is Connes embeddable.
\end{thm}
\begin{proof}
Assume that $\Gamma$ is finite. The only if part is trivial as $\bM_v$ is a subalgebra of $\bM_\Gamma$. Conversely, for every $v \in V\Gamma$ let $\pi_v: \bM_v \rightarrow \prod_{i, \cU} [\bA_{v,i}, \tau_{v,i}]$ be an embedding into an ultraproduct of matrix algebras $\bA_{v,i}$ with tracial states $\tau_{v,i}$. Lemma \ref{Lem=PointGraphProduct} yields an embedding $\bM_\Gamma \rightarrow \prod_{i, \cU} [\bA_{\Gamma,i}, \tau_{\Gamma,i}]$.
Let $\bP_{v,i}$ be a copy of $\bP$ with vacuum state $\tau_{v,i}$. $\bP_{v,i}$ is a II$_1$-factor by Theorem \ref{Thm=GaussianFactor}. Let $\bP_{\Gamma,i}$ be the graph product of these  factors. By Lemma \ref{Lem=MatrixFactor} each $\bA_{v,i}$ is an expected subalgebra of $\bP_{v,i}$ and so $\bA_{\Gamma,i}$ is an expected subalgebra of $\bP_{\Gamma,i}$. Hence we must prove that $\prod_{i, \cU} [\bP_{\Gamma,i}, \tau_{\Gamma, i}]$ is Connes embeddable. In turn by \cite[Lemma 7.14]{Jun} it suffices to prove that each $\bP_{\Gamma,i}$ is Connes embeddable. The latter is Theorem \ref{Thm=PGammaIsEmbeddable}. For infinite $\Gamma$ the result follows from an inductive limit argument, c.f. the proof of \cite[Corollary 2.17]{CasFim}.
\end{proof}

\end{document}